\documentclass[12pt, oneside,reqno]{amsart}   	
\usepackage{geometry}
\usepackage{float}
\allowdisplaybreaks
\usepackage{graphicx}				
\usepackage{amssymb}
\usepackage{amsmath}
\usepackage[all]{xy}
\usepackage{mathtools}
\usepackage{tikz-cd}
\usepackage{enumitem}
\usepackage{tikz-cd}
\usepackage{appendix}

\usepackage[%
bookmarks=true,
colorlinks,
linkcolor=blue,
urlcolor=blue,
citecolor=blue,
plainpages=false,
pdfpagelabels,
final,
breaklinks=true\between
]{hyperref}
\usepackage{hyperref}
\usepackage[numbers,sort&compress]{natbib}

\DeclareMathOperator{\Length}{Length}

\newtheorem{thm}{Theorem}[section]
\newtheorem{lem}[thm]{Lemma}

\newtheorem{defn}{Definition}[section]

\numberwithin{equation}{section}

\newcommand{\ve}{\varepsilon}

\def\Pb{\ifmmode{\Bbb P}\else{$\Bbb P$}\fi}
\def\Z{\ifmmode{\Bbb Z}\else{$\Bbb Z$}\fi}
\def\C{\ifmmode{\Bbb C}\else{$\Bbb C$}\fi}
\def\R{\ifmmode{\Bbb R}\else{$\Bbb R$}\fi}
\def\S{\ifmmode{S^2}\else{$S^2$}\fi}

 \DeclareMathOperator{\proj}{proj}

\def\S{\cal S}

\newenvironment{pf}{\paragraph{Proof:}}{\hfill$\square$ \newline}

\begin{document}

\title[nonplanar ancient CSF in $\R^3$ from reapers]{Nonplanar ancient curve shortening flows in $\R^3$ from grim reapers}
\author{Theodora Bourni and Alexander Mramor}
\address{Department of Mathematics, University of Tennessee Knoxville,
Knoxville TN, 37996
}
\email{tbourni@utk.edu}
\address{Department of Mathematics, Johns Hopkins University, Baltimore MD, 21231}
\email{amramor1@jhu.edu}

\begin{abstract} In this note we construct new nonplanar ancient (in fact, eternal) solutions to the curve shortening flow in $\R^3$, built out of translating grim reapers laying in perpendicular planes. 
\end{abstract}
\maketitle
\section{Introduction}

Ancient solutions to the mean curvature flow (MCF), which for curves is also called the curve shortening flow (CSF), serve as models for singularities which may develop after a rescaling process and are also the natural analogues of complete solutions to the heat equation in submanifold geometry. There are many known examples of ancient CSF/MCF in $\R^n$, many of which are nonplanar, see for instance \cite{AAAW, LT}. The quality of nonplanarity is interesting for instance because there are a number of recent results for the mean curvature in higher codimension where one can constrict ancient solutions to some proper affine subspace of $\R^n$ under some conditions, see for example \cite{Cal, CM, Naff}. In particular, a nonplanar curve shortening flow in $\R^3$ takes up as much ``room'' as possible. The  point of this note is to construct another such example by, in short, attaching grim reapers in different planes along their ends. To summarize the main properties of the construction: 
\begin{thm} There exists a smooth noncompact curve shortening flow $M_{t}$, $t \in (-\infty, 0]$, such that:
\begin{enumerate} 
\item $M_t$ doesn't lay in any affine plane for any $t \in (-\infty, 0]$.
\item In the limit as $t \to - \infty$, $M_t$ converges to three parallel lines, which pairwise are at most distance $\sqrt 2\pi$ apart.
\item $M_t$ is a ramp, in the sense of Altschuler and Grayson (see Definition \ref{ramp}). 
\end{enumerate}
\end{thm} 
The properties listed above correspond to a construction of an ancient solution modeled on combining two grim reapers translating with speed 1, but it will be evident that one can use more grim reapers and additionally vary their widths with appropriate modifications (including modifying item (2) above, where the number of asymptotic lines will increase). The flows we construct will also clearly be extendable to an eternal solution, that is a solution defined for $t \in \R$ although this is perhaps of lesser note; in the limit as $t \to \infty$ it will converge to a line. 
$\medskip$

The path we take, as is often the case in these types of constructions, is to first construct approximating ``old but not ancient'' solutions (often referred to below as simply approximating solutions) $M^i_t$ defined on $(-T_i, 0]$ where $T_i \to \infty$, take a limit of these via curvature estimates, and show the limit is nonempty and furthermore nontrivial (in this case, nonplanar). The approximate solutions are constructed by piecing two grim reapers together, where more or less one lays in the $xy$-plane and the other lays in the $yz$-plane, interpolating in the middle and, as we'll discuss more below, bending the curves far from the origin to make the use of the maximum principle on these noncompact curves simpler; since these regions where we bend are farther and farther away from the origin for each $i$, in the limit they are blown off to spatial infinity. The limit then, as we will show, is an ``ancient trombone'' -- see Figure 1, which can arguably be thought of as a higher codimension analogue of the examples in \cite{AY}. 
$\medskip$

\begin{figure}
\centering
\includegraphics[scale = .45]{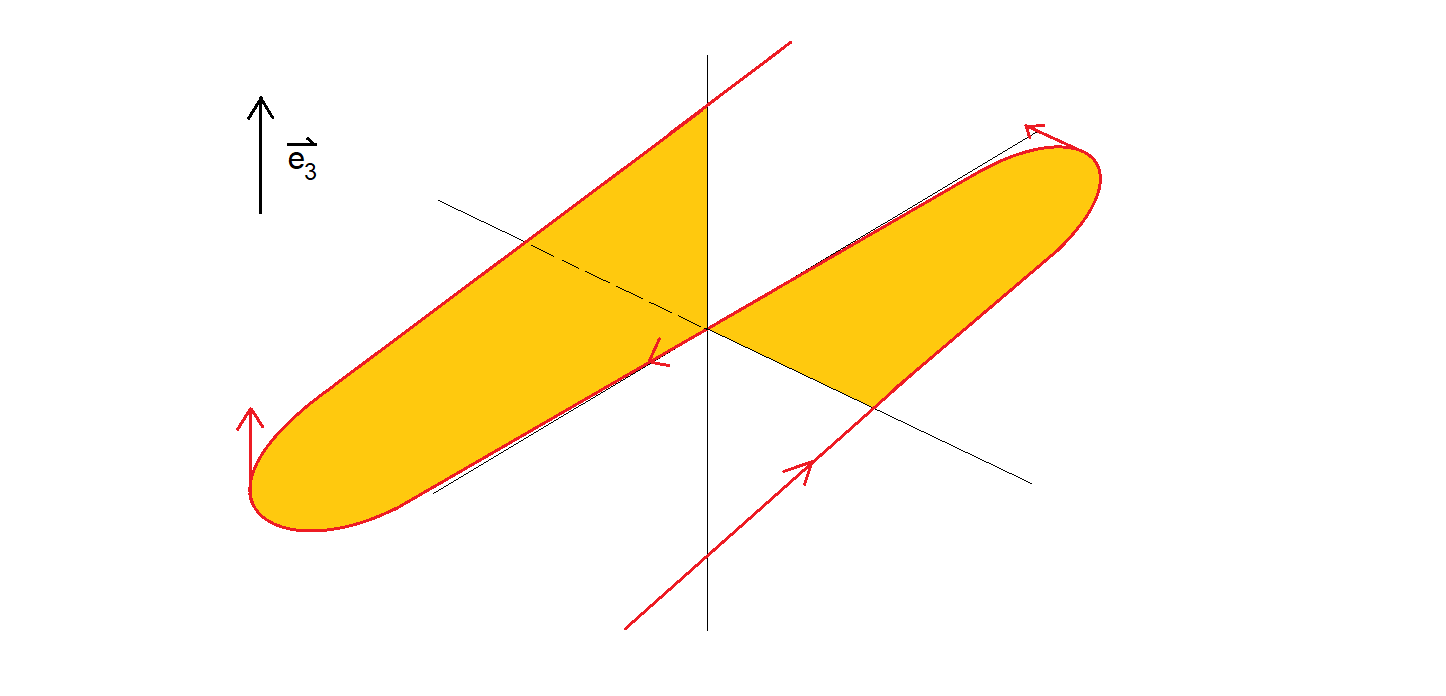}
\caption{A sketch of what the ancient solution looks like. The regions above are shaded just to make the diagram appear less ``flat'', although coincidentially these do essentially represent the regions to which the area estimate below (Lemma \ref{Areaest}) is applied. The arrows  of course indicate the tangent to the curve.}
\end{figure}
$\medskip$

To take a limit, one uses that the approximate solutions are ramps in the sense of Altschuler and Grayson \cite{AG}, a notion that is discussed more below (see Definition \ref{ramp}). By the maximum principle we can then get uniform curvature estimates. It will be easy to see that the limit flow is nonempty but what might be less clear is that it is nonplanar; for instance it might be feared that the limit is a stationary line (indeed, this is the asymptotic behavior one expects as $t \to \infty$), or a union of stationary lines if the ``tips'' of the approximate grim reapers don't move in fast enough. To deal with this, one tool we use is perhaps the lesser known fact that flows of convex hypersurfaces are barriers to the curve shortening flow. Of course for general hypersurfaces this does not hold true, considering for instance a round circle wrapped around the neck of a catenoid. By using what we refer to as Angenent cylinders, taking products of Angenent ovals and $\R$, and constructing the initial data for the approximate solutions with some care, one can ensure the limit doesn't lay in a plane.  This is roughly done by showing that the two orthogonal approximate grim reapers don't ``twist'' out of their initial planes (or perhaps more correctly, one shows that they stay in thin slabs) too much for a very long time. This rules out at least the limit laying in a plane, but it could still be that the limit is a union of lines -- where, geometrically speaking, this would be the case if  the ``tips'' of the attached grim reapers don't move in quickly enough. Naively one would imagine trying to rule this out by using grim planes as ``outer barriers,'' but it seems impossible to arrange them in a way which keeps them disjoint from the initial data. To get around this, we use the detailed asymptotics of grim reapers and an area estimate method using Gauss-Bonnet (one instance where the ambient dimension $n =3$ is used) to show that the tips move in quickly, and thus showing that the ancient solution constructed is indeed a legitimate nonplanar ancient curve shortening flow -- i.e. not just a union of ancient planar CSFs laying in different planes. 
$\medskip$

One imagines that related compact constructions are possible, by ``tying'' up all the ends using additional grim reapers (it's easy to come up with plausible candidates). Such examples wouldn't be ramps though, so modifications to our argument would be necessary. Because our examples have simple blowdowns, one might imagine with some modifications/generalizations one may be able to use them to obtain ``good'' bounds for the constants in the statement of Corollary 0.6 of \cite{CM}, which give condimension bounds for ancient flows in terms of their entropy. 
$\medskip$ 

$\textbf{Acknowledgements:}$ The authors thank Jacob Bernstein for pointing out the possible applications to \cite{CM} mentioned above, as well as Mat Langford for helpful discussions on constructions of related ancient flows. At the time of writing T.B. was supported by grant NSF-DMS 2105026 and 707699 Simons collaboration grant, and A.M. by an AMS-Simons travel grant; we thank them for their support. We also thank the anonymous referee for their careful reading and comments. 

\section{Preliminaries}

Before actually describing the approximate solutions in more detail, we first lay out the tools we need to control these solutions, which naturally put constraints on how they should be constructed. The first facts and notions we discuss are taken from section 2 of Altschuler and Grayson's paper \cite{AG}. The first definition is a natural generalization of graphicality in higher condimension.
$\medskip$

\begin{defn}\label{ramp} Let $\gamma \subset \R^3$ be an embedded, arclength parameterized, curve with tangent vector $T$. Then $\gamma$ is said to be a ramp if $\langle T, V \rangle \geq 0$ for some vector $V$. 
\end{defn} 

The evolution equation for $\langle T, V \rangle$ is given by the following, where $\kappa$ is the geodesic curvature and $s$ is the arclength parameter.
\begin{equation}
\frac{\partial}{\partial t} \langle T, V \rangle = \frac{\partial^2}{\partial s^2} \langle T, V \rangle + \kappa^2  \langle T, V \rangle\,. 
\end{equation} 
Even though there are noncompact maximum principles available, in the argument below we will construct our approximate solutions in a way that ensures $ \langle T, V \rangle $ is \emph{very} positive and stays so for a long time by pseudolocality far away from the origin. Therefore, one may apply the maximum principle as in the compact case to see that this quantity stays positive and in particular its minimum doesn't decrease. The following evolution equation gives that lower bounds of $ \langle T, V \rangle$ along with initial upper bounds on $\kappa$ give bounds on $\kappa$ in later times.
\begin{equation} \label{ktv}
\frac{\partial}{\partial t} \frac{\kappa}{ \langle T, V \rangle} = \frac{\partial^2}{\partial s^2}\frac{\kappa}{ \langle T, V \rangle} + 2 \frac{2}{ \langle T, V \rangle} \frac{\partial}{\partial s} \langle T, V \rangle \frac{\partial}{\partial s}\frac{\kappa}{ \langle T, V \rangle} - \frac{\kappa}{ \langle T, V \rangle} \tau^2\,,
\end{equation} 
where $\tau$ here is the torsion.
Again, we will construct our approximating solutions in such a way to ensure that, from \eqref{ktv}, $\frac{\kappa}{ \langle T, V \rangle}$ must be non increasing by the classical maximum principle (so without having to resort to noncompact ones). 
$\medskip$

In the construction, the following simple but important ``area estimate'' will be used to essentially give a lower bound on the speed of the tips of the approximate solutions, and thus ensuring that the limiting flow is not just a union of lines. Below, let $\{\tilde\Gamma_t\}_{t\in [0, \infty)}$ be a curve shortening flow in $\R^3$ let $\Gamma_t$ be a compact piece of $\tilde \Gamma_t\cap\{y<0\}$, let $A_t, B_t\in \{y=0\}$ be the endpoints of $\Gamma_t$ and let $\ell_t$ be the line segment joining them to create, along with $\Gamma_t$, a closed piecewise smooth curve $D(t)$:

\begin{lem}\label{Areaest} Let $A(t)$ denote the area of a minimal surface bounded by $D(t)$, with $D(t)$ as above, $t\in [0, T)$. Then, if the interior angles at $A_t, B_t$ are bounded by $\pi +\ve(t)$, we have
\[
 \frac{d A(t)}{dt}\le -\pi+\ve(t)\,.
\]
\end{lem}
\begin{proof}
Note first that
 \[
 \frac{d A(t)}{dt}=-\int_{\Gamma_t}\kappa_g
 \]
where $\kappa_g$ is the geodesic curvature of the boundary curve, which is equal to the curvature vector of $\Gamma_t$ dotted with the conormal of the minimal disk at the boundary. By Gauss-Bonnet (minimality used here):
 \[
 2\pi\le \int_{\Gamma_t}\kappa_g+\pi+\ve(t)
 \]
 and therefore
  \[
 \frac{d A(t)}{dt}=-\int_{\Gamma_t}\kappa_g \le -\pi+\ve(t)\,.
 \]
\end{proof}

To employ this area estimate, we will need pretty tight control of the approximate solutions we construct; the following lemmas and facts are involved in this control. To begin, the following lemma says that we can use convex two dimensional mean curvature flows as barriers for curve shortening flows:

\begin{lem}\label{convbarrier} Let $\Gamma_t$ be a curve shortening flow and $M^2_t$ a convex solution to mean curvature flow which are initially disjoint, and which in later times, $t \in [0, T)$, $T \leq \infty$, might only intersect in a bounded region. Then, in fact, they remain disjoint for $t \in [0,T)$. 
\end{lem} 

\begin{proof}
Let $P$ be any affine plane and note that $P\cap M_t$, as long as it is not empty, is a subsolution to CSF in P (moves faster than curvature), since the other curvatures in transverse directions are positive by convexity -- for a more general statement and detailed argument see Lemma 2.3 in \cite{BLL}. So assume that there is a first time $t$ when $\Gamma_t\cap M_t\ne \emptyset$ and let $p\in \Gamma_t\cap M_t$. Let $P$ be the osculating plane of $\Gamma_t$ at $p$. Then the curvature vector of $\Gamma_t$ at $p$ coincides with the curvature (times the normal) of the projection of $\Gamma_t$ on P. Therefore, at $p$, $\proj_P(\Gamma_t)$  moves slower then $P\cap M_t$, which contradicts the fact that $t$ is the first time that $\Gamma_t$ and $M_t$ intersect.
\end{proof}

Now we discuss some already known solutions to the mean curvature flow which play important role in our construction. As is certainly clear, grim reaper translators play a central role in our construction. The grim reaper $G$ of width $\pi$ and speed 1 is given by the graph of  $y =  \ln \cos{x}$, where $x$ ranges between $\pm \frac{\pi}{2}$; we say it has width $\pi$ because it is asymptotic to two parallel lines distance $\pi$ apart (in this case, the lines $x = \pm \frac{\pi}{2}$). Its flow $G_t$ exists for all time and translates downward with speed 1 along the $y$-axis, given by the graph $y = \ln \cos{x} - t$. Closely related to grim planes (products of grim reapers with lines) are the Angenent cylinders, which we will extensively use as barriers. The  Angenent cylinders $A_t$ are solutions of the form $a_t \times \R$, where $a_t$ denotes the Angenent oval of width $\pi$, $t \in (-\infty, 0]$, given explicitly by $\{\cos x= e^t\cosh y\}$. In the following lemma we summarize its relevant properties.
$\medskip$

\begin{lem}\label{ovalfacts} Let $A_t$ be the Angenent cylinder oriented so that $a_t$ is in the $xy$ plane with its semi-major axis laying along the $y$-axis (see also section 2 of \cite{BLT}).
\begin{enumerate}
\item $A_t$ is convex. 
\item As $t \to -\infty$, $a_t$ converges to two opposite facing grim reapers of width $\pi$ translating, in opposite directions,  along the $y$-axis. 
\item Denote by $d_S (A_t)$ the distance of $A_t$ from the boundary of the slab with sides $x = \pm \frac{\pi}{2}$. Then $d_S(A_t) < 2e^{t}$, for any $t\in (-\infty, 0)$.  
\end{enumerate} 
\end{lem} 

 The following lemma will be used to control the contribution from the interior angles of the surface on which we will apply the area estimate above (to the shaded regions in Figure 1). In practice, the lemma will apply because we'll have uniform curvature bounds along the flow which we can then scale to be bounds by 1.
\begin{lem}\label{smallangle} Suppose $\gamma$ is a length parameterized curve whose curvature is bounded by $1$ contained in the intersection of two orthogonal slabs of width $\epsilon > 0$ which contains the $y$-axis. Then $|\langle T, v \rangle | <   3\sqrt{\epsilon}$, where $T$ is the unit tangent vector to $\gamma$ and $v$ is any unit vector perpendicular to $e_2$. 
\end{lem}

\begin{pf} 

Considering such a curve $\gamma$, fix a point $p$ on it and suppose there is a unit vector $v$ (replacing $v$ with $-v$ if necessary) for which $\langle T(p), v \rangle  >   3\sqrt{\epsilon}$. Without loss of generality, $v = e_1$, coresponding to the $y$ coordinate. By integration and using the curvature bound, at all points less than distance $\sqrt{\epsilon}$ further along the curve we have $\langle T, e_1 \rangle  >   2\sqrt{\epsilon}$. Integrating again then, we get that following along $\gamma$ distance $\sqrt{\epsilon}$ from $p$ implies there is a point $q \in \gamma$ with $\langle q, e_1\rangle  - \langle p, e_1\rangle > 2\epsilon > \sqrt{2} \epsilon$, and hence must be a point which lays outside the slab intersection, giving a contradiction. 
\end{pf}

\section{The construction} 

As discussed in the introduction we begin by constructing the approximate solutions. In their description below, note that there are arbitrary choices made; one probably expects these choices do not have a discernible effect on the limit ancient flow. We first focus on discussing how to construct approximate solutions which are ramps in the sense discussed in section 2 and in particular satisfy (uniform) bounds on $\kappa/ \langle T, V\rangle$ for an appropriately chosen vector $V$ -- which is what gives us curvature estimates. Then we discuss how to arrange the initial data appropriately, in particular the parameter $R$ we introduce below, to get a nontrivial ancient flow.
$\medskip$

First, we give a preliminary definition of the initial data for the approximate solutions, which will then be modified accordingly. Consider the grim reaper $\{G_t\}_{t\in(-\infty, \infty)}$ in the $xy$-plane moving with speed 1 in the direction of $e_2$, and let $(\pm x_R, 0)$ be the intersection of $G_{-R}$ and the $x$-axis. Now we define $P_{R}=(G_{-R}-(x_R, 0)) \cap\{(x, y)| x\le 0\}$ (the first set is a translation of the grim reaper) and let
\[
Q_{R}= (ROT^{y}_{\pi/2}\circ Rot_\pi)(P_R)
\]
where $Rot_\pi$ is a  rotation by $\pi$ in the $xy$-plane around the origin and $ROT^y_{\pi/2}$ is a counterclockwise rotation by angle $\pi/2$ in $\R^3$ around the $y$-axis so that $Q_R\subset \{(x, y, z)| x=0, y\le R, z\ge 0\}$. Then our preliminary version of the initial data for the approximating solutions (parameterized by $R$) is taken to be the curve $\Gamma_R = P_R\cup Q_R$, or in other words two grim reapers with one of their ends cut off and spliced onto the other roughly as indicated in Figure 1 in the introduction. 
$\medskip$

Our goal of controlling $\langle T, V \rangle$ and $\kappa/ \langle T, V \rangle$ appropriately on the initial data of course depends on a good choice of vector $V$, which we discuss next, and which leads naturally to how to refine the preliminary initial data for the approximate solutions.
$\medskip$

Note that $G_{-R}$, since its a time slice of a translator, satisfies $\kappa=-\langle \nu, e_2\rangle =\langle T, e_1\rangle$; indeed this is a key reason why one might imagine being able to control $\kappa/ \langle T, V \rangle$ for some choice of $V$. From this note that on $P_R$ and $Q_R$ respectively we have
\[
P_R: \kappa=\langle T, e_1\rangle >0\,,\,\, Q_R: \kappa=\langle T, -e_3\rangle <0\,.
\]
With this in mind, consider the vector $e=\frac{\sqrt 2}{2} e_1+\frac{\sqrt 2}{2} e_3$. Then, on $P_R$
\[
\langle T, e\rangle =\frac{\sqrt 2}{2}\langle T, e_1\rangle = \frac{\sqrt 2}{2}\kappa>0\,,
\]
and on $Q_R$
\[
\langle T, e\rangle =\frac{\sqrt 2}{2}\langle T, e_3\rangle = -\frac{\sqrt 2}{2}\kappa>0\,.
\]

From these calculations, we see that $\Gamma_R$ almost satisfies being a ramp with $V = e$, although we need to be careful at the point $I$ where we attached $P_R$ and $Q_R$. At $I$, $\Gamma_R$ is not smooth as the left and right limits of $T$ differ although only by $e^{-R}$ up to a scale factor. Naturally one would wish to mollify about the point $I$, but one might worry that control on $\kappa/\langle T, V\rangle$ could be lost, so we take a more geometric route. Note that one may slightly bend $Q_R$ slightly along the vector $e_1 - e_3$ near $I$, and that any amount of bending will make the curve enter the quarterspace $\{(x,y,z) \mid x > 0, z < 0\}$. Similarly we may slightly bend $P_R$ into $\{(x,y,z) \mid x < 0, z > 0\}$. From the direction of the bending we see that the lower bound on $\langle T, V \rangle$ will be preserved, and by bending slightly enough we may arrange that $\kappa/\langle T, V\rangle \le 2$ holds. After a (slight) bending, we may consider a slight ``upward'' translation $\widetilde{Q_R}$ of $Q_R$ in the $yz$ plane and a slight translation $\widetilde{P_R}$ of $P_R$ in the $xy$ plane (note this doesn't affect the above inequalities) so that the ends of the bent curves match, and so that the resulting curve is smooth with
\[
\langle T, V\rangle>0\,,\,\, \frac{\kappa}{\langle T, V\rangle}\le 2\,.
\]
By abusing notation, we still use $\Gamma_R$ to denote the modified smooth curve. Note, since the deformation above can be taken to be of the order of $e^{-R}$, for any $\Gamma_R$ we may arrange it to be close in Hausdorff distance to $P_R \cup Q_R$ of the order of $e^{-R}$. Now at this point the initial data for the approximate solutions are well-controlled ramps (with $V = e$), and we have to ensure that they do stay ramps into the future, which has some complications due to their noncompactness. Because the terms involved tend to zero at the ends of $\Gamma_R$, there appear to be issues with applying a noncompact maximum principle though (at least the well known ones).
$\medskip$

Instead, we bend/flare out the ends of $\Gamma_R$ (and relabel back), so that it is asymptotic to two lines $\ell_1$ and $\ell_2$ which satisfy $\langle T, V\rangle =1$ on them. By continuity and geometric reasoning we can also arrange that $\langle T, V\rangle > 0$ is preserved where the bending occurs. To set notation for the sequel we may arrange furthermore that $\Gamma_R$ above is left unperturbed within the ball of radius $R_{bend} >> R$, where we can choose $R_{bend}$ to be as large as we wish. Considering the evolution $(\Gamma_R)_t$ of $\Gamma_R$ under the flow, by pseudolocality \cite{CY}, we may arrange that  $1-\langle T, V\rangle $ and $\kappa/\langle T, V\rangle$ is as small as we wish sufficiently far from the origin, the distance one must go depending on $t$. This gives that $\kappa/\langle T, V\rangle\le 2$ (with the bending done gradually) is preserved along the flow by the standard parabolic maximum principle. In particular, the flow of $\Gamma_R$ is smooth for all time. 
$\medskip$

With these $(\Gamma_R)_t$ in hand we recall that to construct a nontrivial (in the sense as discussed in the introduction) ancient flow, it will suffice to produce a sequence $M^i_t$ of flows such that
\begin{enumerate} 
\item the flows $M^i_t$ are defined on time intervals $[-T_i, 0]$ with $T_i \to \infty$.
\item $|\kappa| < C$, for a uniform constant $C$ independent of time, on the $M^i_t$.
\item There is a uniform ball $B(0, \overline{R})$ and time $t$ so that $M^i_{t} \cap B(0, \overline{R})$ are all connected curves and ``far'' from being planar in that there is a uniform positive lower bound in Hausdorff distance from them to any plane intersected with $B(0,\overline{R})$. 
\end{enumerate} 
Of course, item (2) implies (1); as we'll obviously construct the $M^i_t$ in terms of the $(\Gamma_R)_t$ the main point will be to check that we may arrange item (3) for a sequence with $T_i \to \infty$. 
$\medskip$

We claim, that if we pick $R_i = 10^i$, then the following is true
\begin{itemize}
\item[(a)] there is some $\overline{R} > 0$ for which  $\partial B(0, \overline{R}) \cap (\Gamma_{R_i})_{10^i - 1000}$ consists of two points (corresponding to the two asymptotic lines essentially, by taking $R_{Bend}$ sufficiently large which we may), and 
\item[(b)] $B(0, 100) \cap (\Gamma_{R_i})_{10^i - 1000}$ is distance less than 1/10 to the concatenation of two curves laying in the $xy$ and $yz$ planes, which have points distance at least 9/10 away from the line $\{y = 0\}$ (of course these are essentially the evolutions of $P_R$ and $Q_R$). 
\end{itemize} 
Item (b) shows that $ (\Gamma_{R_i})_{10^i - 1000}$ is nonplanar, and (a) shows that it is connected. The second statement (b) will be evident from the proof of (a), because of our use of barriers.
$\medskip$

To see (a), we consider the compact pieces of $(\Gamma_{R_i})_t\cap \{\pm y>0\}$ and let  $(p^t_1, p^t_2)$ and $(q^t_1, q^t_2)$ be the corresponding endpoints (note that these are in fact three distinct points). Let now $D^{1,i}_t$ and $D^{2,i}_t$ be the two piecewise smooth curves formed by the two compact pieces of the curve along with two line segments joining the endpoints (essentially, the boundaries of the orange shaded regions in Figure 1).
$\medskip$

To control the interior angles at the corner points we will use Lemma \ref{convbarrier} with Angenent cylinders as barriers for $(\Gamma_{R_i})_t$ by arranging them to form boundaries of slabs for use in Lemma \ref{smallangle} as indicated in Figure 2; by combining it with Lemma~\ref{ovalfacts} (which controls how quickly the slab intersection will widen) we can ensure the function $\ve(t)$ in Lemma \ref{Areaest} to be bounded by $2e^{-(R_i - t)/2}$; in particular it is integrable. More specifically, to control each corner point we use 6 Angenent cylinders $A_{t-R_i}$, with 2 ``threaded'' through the $P_{R_i}$ and $Q_{R_i}$, and the other four placed opposingly, so that they ``enclose'' $P_{R_i}$ and $Q_{R_i}$. Because the ends of the $\Gamma_{R_i}$ are flared clearly these can be arranged to be disjoint from the initial data, and by Lemma \ref{convbarrier} they will remain disjoint for all later times. Because the initial areas of the minimal disks bounded by $D^{1,i}_t$ and $D^{2,i}_t$ are up to a uniform additive constant $\pi R_i$, by integrating the estimate of Lemma \ref{Areaest} we find that at $t = 10^i - 1000$ the area of the minimal disk enclosed by $D^{1,i}_t$ and $D^{2,i}_t$ is bounded by a constant independent of $i$. 
 $\medskip$
 
 \begin{figure}
\centering
\includegraphics[scale = .45]{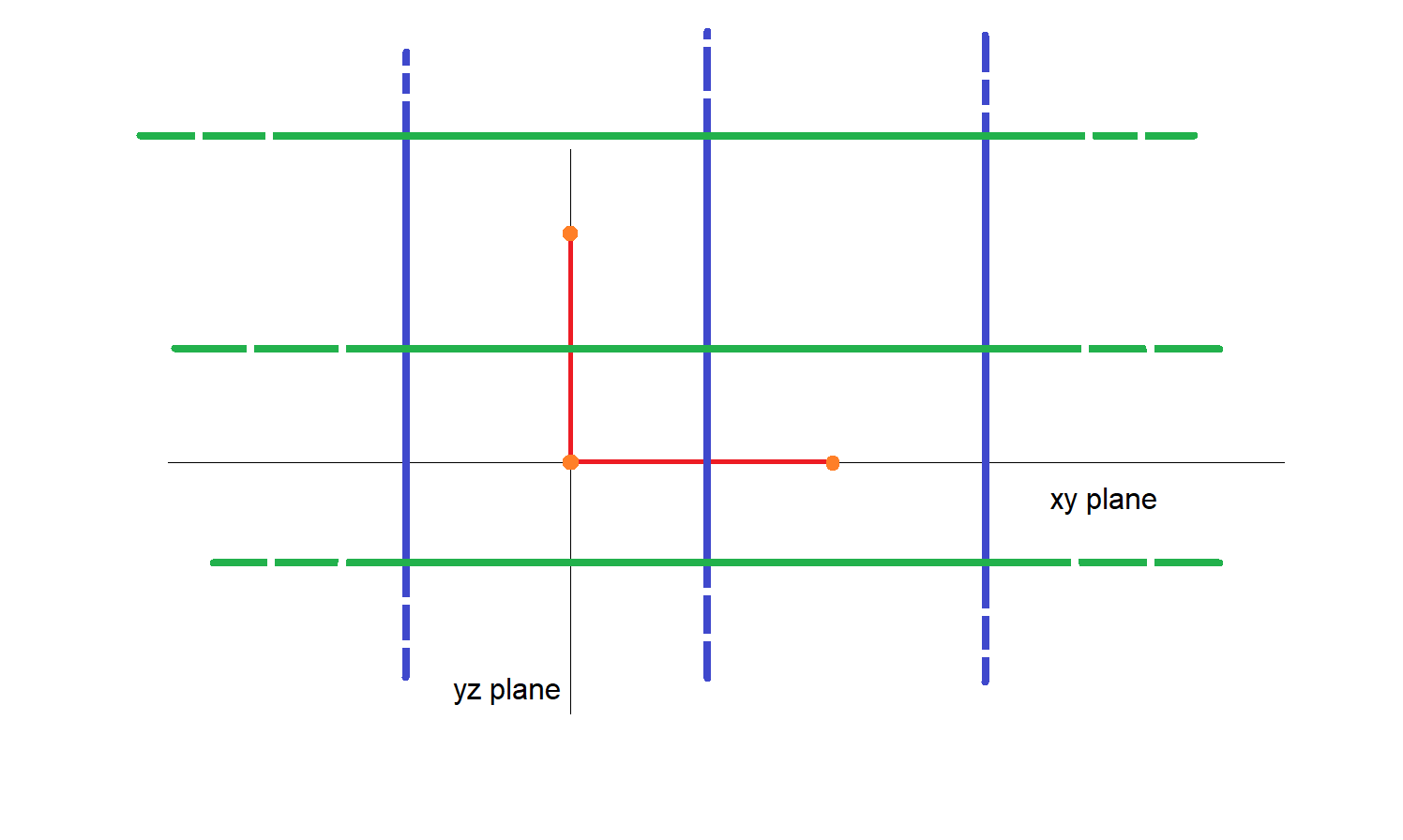}
\caption{A side view of the configuration of the barriers, looking along the $y$-axis. Here the red ``L'' represents the initial data (the initial data can't lay precisely in the planes, but we can arrange it to be as close as we wish), the orange dots indicate approximately where the points $(p_1, p_2)$ and $(q_1, q_2)$ would be (drawn as one expects, that two of the points are the same or at least very close to each other), and the blue and green lines indicate the $\R$ factors of the Angenent cylinders. Note each of the orange points lay in a slab intersection defined by these cylinders. }
\end{figure}

 We now claim that this area estimate implies that at $t = 10^i - 1000$ the minimal disks bounded by $D^{1,i}_t$ and $D^{2,i}_t$ are both contained in $B(0, \overline{R})$ for some large uniform choice of $\overline{R}$. To see this we will show that the curves $ D^{1,i}_t$ and $  D^{2,i}_t$ have a canonical tubular neighborhood for $t \in [0, 10^i - 1000]$ of uniform size (i.e. independent of $t$ and $i$). Intersecting this neighborhood with the minimal disks bounded by $D^{1,i}_t$ and $D^{2,i}_t$ we find that there is a uniform $c$, such that $c\Length(D^{1,i}_t)$ and $c\Length(D^{1,i}_t)$ are bounded by the areas of the minimal disks bounded by $D^{1,i}_t$ and $D^{2,i}_t$ respectively, giving the desired bound. 
 $\medskip$
 
 To see that this is true, first note by considering Angenent cylinders translated so that they are centered by/closer to the ``tips'' of the initial data, we can arrange for any $\delta > 0$ that $ D^{1,i}_t$ and $  D^{2,i}_t$ are both contained in slabs of width $\delta$ for $t \in [0, 10^i - 1000]$ for $i$ large enough, potentially changing $10^i - 1000$ by a uniform additive constant. By the curvature bounds on the $(\Gamma_{R_i})_t$, this implies arguing exactly as in Lemma \ref{smallangle} that $ D^{1,i}_t$ and $ D^{2,i}_t$ are graphical over the $xy$ and $yz$ planes respectively, with bounded curvature away from the (projections of the) corner points $p_1^t, p_2^t, q_1^t, q_2^t$. Denote the projections of their boundaries onto these planes by $\sigma^{1,i}_t$, $\sigma^{2, i}_t$. 
 $\medskip$

From our discussion on the choice of $V$, note that the $\Gamma_{R_i}$ are also simultaneously ramps with respect to the vectors $ae_1 + be_3$ for any $a, b \geq 0$ and that this will be preserved under the flow (although we see if initially we chose in defining $V$ to have $a$ or $b$ to be zero this would not lead to curvature bounds -- this is not a concern here though). This gives that on $\sigma^{1,i}_t$ we have $\langle T, e_1 \rangle \geq 0$ and on $\sigma^{2,i}_t$ we have $\langle T, e_3 \rangle \geq 0$ (so are ramps in these planes). Clearly to show $ D^{1,i}_t$ and $ D^{2,i}_t$ have tubular neighborhoods it suffices to show $\sigma^{1,i}_t$, $\sigma^{2, i}_t$ do. Now we see that $\sigma^{1,i}_t$, $\sigma^{2, i}_t$ are contained individually in slabs of width little more than $\pi$ in their respective planes. Also by the curvature bounds we see the only way these curves don't have uniform tubular neighborhoods are if for each $\ve > 0$ there is an $i$ and corresponding $t_i$ for which on $\sigma^{1,i}_{t_i}$ or $\sigma^{2,i}_{t_i}$ we have two sheets (that is, locally graphical regions) which are distance less than $\ve$ apart. Then, there are two such sheets with  tangent vectors pointing in opposite directions, so by the conditions $\langle T, e_1 \rangle > 0$ and $\langle T, e_3 \rangle > 0$ there must be a point in between these two sheets along the curve with large curvature depending on $\ve$. Taking $\ve$ small enough violates the curvature bounds along the $(\Gamma_{R_i})_t$, giving us what we want. 
 $\medskip$

 Time translating the flows $(\Gamma_{R_i})_t$ by $-T_i = 10^i - 1000$ gives us the ``old-but-not-ancient'' solutions $M^i_t$ with properties (1)-(3), completing the construction.

\end{document}